\documentclass[10pt,letterpaper]{article}
\usepackage[latin1]{inputenc}
\usepackage{amsmath}
\usepackage{amsfonts}
\usepackage{amssymb}
\usepackage{amsthm}
\usepackage{cite}
\usepackage{xcolor}

\title{Hydra group doubles are not residually finite}
\author{Kristen Pueschel}

\theoremstyle{plain}
\newtheorem*{thma}{Theorem}
\newtheorem*{clma}{Claim}
\newtheorem{thm}{Theorem}[section]
\newtheorem{prop}[thm]{Proposition}
\newtheorem{lemma}[thm]{Lemma}
\newtheorem{corollary}[thm]{Corollary}

\theoremstyle{definition}
\newtheorem{definition}{Definition}

\theoremstyle{remark}
\newtheorem{remark}{Remark}
\usepackage[left=2cm,right=2cm,top=2cm,bottom=2cm]{geometry}
\begin{document}
\maketitle
\begin{abstract} In 2013, Kharlampovich, Myasnikov, and Sapir constructed the first examples of finitely presented residually finite groups with large Dehn functions. Given any recursive function $f$, they produce a finitely presented residually finite group with Dehn function dominating $f$. There are no known elementary examples of finitely presented residually finite groups with super-exponential Dehn function. Dison and Riley's hydra groups can be used to construct a sequence of groups for which the Dehn function of the $k ^{\mbox{\tiny{th}}}$ group is equivalent to the $k ^{\mbox{\tiny{th}}}$ Ackermann function. Kharlampovich, Myasnikov, and Sapir asked whether or not these groups are residually finite. We show that these constructions do not produce residually finite groups.

Classification: 20E26, 20E06
\end{abstract}

\begin{section}{Introduction}

The first examples of finitely presented residually finite groups with super-exponential Dehn function were constructed in \cite{KMS}:

\begin{thma}[{Kharlampovich, Myasnikov, and Sapir}] For any recursive function ${f: \mathbb{N} \to \mathbb{N}}$, there is a finitely presented residually finite solvable group $G$ of derived length 3 for which the Dehn function ${\delta_G \succcurlyeq f}$.
\end{thma} 

Their examples are sufficiently complicated that it remains interesting to find elementary examples that arise `in nature'. One place to look is among known elementary examples of groups with large Dehn function. In \cite{Tim}, Dison and Riley introduced the hydra groups $$G_k := \langle a_1, \dots, a_k, t~|~ a_1^t= a_1, a_i^t=a_ia_{i-1},~ i > 1 \rangle,$$ where we use the conventions $a^b:= b^{-1}ab$ and $[a,b]=a^{-1}b^{-1}ab$. They proved that the HNN extension $$\Gamma_k = \langle G_k, p~|~[a_it,p]=1,~ \forall ~ 1\leq i \leq k \rangle$$ over the subgroup $H_k=\langle a_1t,\dots, a_kt\rangle$ has Dehn function equivalent to the Ackermann function $A_k(n)$.

In \cite{KMS}, the authors commented that it was unknown whether or not $\Gamma_k$ is residually finite for all $k>1$, but that they expected $\Gamma_k$ would not be residually finite for $k>1$. We confirm this.

\begin{thm}\label{notrf1} For all $k>1$, the group $\Gamma_k$  is not residually finite.
\end{thm}

The free product with amalgamation
$$\Gamma^{\prime}_k= \langle G_k, \overline{G_k} ~|~ a_it=\overline{a_it},~~1\leq i \leq k \rangle$$ also enjoys a fast growing Dehn function. An analogous theorem holds for these groups:

\begin{thm}\label{notrf}
	For all $k>1$, the group $\Gamma^{\prime}_k$ is not residually finite.
\end{thm}

\begin{definition} The subgroup $H \leq G$ is separable in $G$ if for all $g \not\in H$, there is a finite quotient $\phi: G \to Q$ such that $\phi(g) \not \in \phi(H)$. Equivalently, $H \leq G$ is separable if and only if it is closed in the profinite topology of $G$, which means that $$H= \bigcap_{\substack{H<K<G \\ \left[G:K\right] < \infty}}K.$$
\end{definition}

In the next section we will see that separability of the subgroup $H_k$ in $G_k$ is necessary for the residual finiteness of $\Gamma_k$ and $\Gamma^{\prime}_k$. Therefore, Theorems \ref{notrf1} and \ref{notrf} are proven via:
\begin{lemma}\label{notsep}
The group $H_k$ is not separable in $G_k$ for any $k>1$.
\end{lemma}

In particular, we will show that the non-separability of $H_2$ in $G_2$ implies non-separability of $H_k$ in $G_k$. To see that $H_2$ is not a separable subgroup of $G_2$, we recognize $(G_2,H_2)$ as isomorphic to an important group-subgroup pair $(G_{BKS}, H_{BKS})$ studied by Burns, Karrass, and Solitar in \cite{BKS}. Burns, Karrass, and Solitar proved that $H_{BKS}$ is a non-separable subgroup of $G_{BKS}$. The group $G_{BKS}$ was the first example of a 3-manifold group containing a finitely presented non-separable subgroup \cite{BKS} and it has been an important tool for verifying other examples of non subgroup-separable groups. For example, Niblo and Wise showed that $G_{BKS}$ virtually embeds in the fundamental group of the complement of the link of 4 circles, $L$. Therefore, $L$ is not subgroup separable \cite{NibWise}. Further, they showed that the fundamental groups of compact graph manifolds have only one obstruction to subgroup separability: the existence of an embedding of $L$ (and hence a virtual embedding of $G_{BKS}$). Niblo and Wise have also shown that $G_{BKS}$ contains finitely presented subgroups which are contained in no proper finite-index subgroups. That is, there is a proper subgroup $K$ such that finite quotients of $G_{BKS}$ will not witness that $K$ is a proper subgroup.\cite{NibWise2}.\\

Dison and Riley have constructed variations on their group-subgroup pairs that also have large distortion. These too can be used to produce candidates for elementary examples of finitely presented groups with fast-growing Dehn functions that might be residually finite. For ${\mathbf{w}= (w_1, \dots, w_k)},$ where $w_i$ is a positive word on letters in ${\{a_1, \dots, a_{i-1}\}}$, consider the group $${G_k(\mathbf{w})= \langle a_1,\dots, a_k ~|~ a_i^t=a_iw_i, ~1\leq i \leq k \rangle}.$$ For powers ${\mathbf{r}= (r_1, \dots, r_k)}$, where ${r_i \geq 0}$, consider the subgroup $${H_k(\mathbf{r}) = \langle a_1t^{r_1}, \dots, a_kt^{r_k}\rangle}.$$ We prove that these cannot be used to produce residually finite groups with large Dehn function. In particular:
\begin{thm}\label{genHydra} $H_k$ is a separable subgroup of $G_k(\mathbf{w})$ if and only if $\mathbf{w}= (1, \dots, 1)$. Therefore the HNN extension $\Gamma_k(\mathbf{w}) = \langle G_k(\mathbf{w}), p~|~h^p=h, h \in H_k\rangle$ is residually finite only if $G_k(\textbf{w})= F_k \times \mathbb{Z}$.
\end{thm} 

\begin{thm}\label{mostgenHydra} Suppose that $\mathbf{w}=(1,\dots, 1, w_c, \dots, w_k)$ where $w_c \neq 1$ and $\mathbf{r}= (r_1, \dots, r_k)$. Let $[w_c]_i$ denote the index of $a_i$ in $w_c$. If $$\sum_{i=1}^{c-1}[w_c]_i r_i \neq 0 $$ then $H_k(\mathbf{r}) $ is a non-separable subgroup of $G_k(\mathbf{w})$.
 \end{thm}

\begin{corollary}
$H_k(\mathbf{r})$ is a separable subgroup of $G_k$ if and only if $\mathbf{r}= (0, \dots, 0)$, that is, the subgroup is separable only in the obvious case that $H_k(\mathbf{r})= \langle a_1, \dots, a_k \rangle$.
\end{corollary}

\begin{remark}The case where $\sum_{i=1}^{c-1}[w_c]_i r_i = 0$ is not understood. In particular, we do not know whether or not $H_3(\mathbf{r})\leq G_3$ is separable for $\mathbf{r}=(0, 1, 0)$.
\end{remark}

We conclude that an example of a residually finite group with super-exponential Dehn function is unlikely to be found as an HNN extension over a subgroup of a hydra-like group.\\

The failure of residual finiteness for the groups $\Gamma_k$ and $\Gamma^{\prime}_k$ leads us to ask if the word problem for $\Gamma_k$ and $\Gamma^{\prime}_k$ is solvable. After all, residual finiteness of the group $G$ always provides a solution to the word problem for $G$ via McKinsey's Algorithm \cite{McKinsey}. Given a word $w$ in the generators of $G$, the algorithm runs two processes in parallel: one lists trivial words, looking for $w$, and the other lists homomorphisms to finite groups, looking to see if $w$ ever has non-trivial image. Our result shows that there are non-trivial elements of $\Gamma_k$ and $\Gamma_k^{\prime}$ for which finite quotients cannot be used to distinguish them from the identity element. Still, the word problems for $\Gamma_k$ and $\Gamma_k^{\prime}$ are decidable. Indeed, Dison and Riley showed that the distortion of $H_k$ in $G_k$ is bounded above by a recursive function, which implies the Membership Problem for $H_k$ is decidable, and therefore that the Word Problem for $\Gamma_k$ and $\Gamma_k^{\prime}$ is decidable \cite{Tim}. In fact, Dison, Einstein, and Riley have found a polynomial time solution to the Word Problem \cite{DER}. \\

I wish to thank my advisor, Tim Riley, for his help, suggestions, and corrections, and Mark Sapir, for a helpful conversation.
\end{section}

\begin{section}{Preliminaries}  

\begin{definition}
A group $G$ is \textit{residually finite} if every element ${x \in G-\{1\}}$ has a non-trivial image in some finite quotient of $G$.  Equivalently, the intersection of all finite index subgroups of $G$ is trivial.\\
\end{definition}
\begin{lemma}\label{7} If ${\Gamma= \langle G, p ~|~h^p=h, ~ h \in H \rangle}$ is residually finite, then $H$ is separable in $G$ and $G$ is residually finite.
\end{lemma}
\begin{proof} It is obvious that $G$ is residually finite, since residual finiteness is inherited by subgroups. Suppose that $H$ is not separable. We can find $g \in G- H$ such that for every homomorphism $\phi$ from $G$ to an arbitrary finite group, ${\phi(g) \in \phi(H)}$.  Consider an arbitrary map $\Phi: \Gamma \to Q$ for $Q$ finite. Since $\Phi$ restricts to a homomorphism on $G$, $\Phi(g)= \Phi(h)$ for some $h \in H$, and $${\Phi([p,g])=[\Phi(p), \Phi(g)]=[\Phi(p), \Phi(h)]= \Phi([p, h])=\Phi(1)=1}.$$ Although ${[p, g]}$ is non-trivial in the free amalgamated product $\Gamma$, it is trivialized in every finite quotient of $\Gamma$. Therefore, if $H$ is not separable in $G$, 
$\Gamma$ is not residually finite.
\end{proof}

By a theorem of Baumslag and Tretkoff, this necessary condition is actually sufficient: if $G$ is residually finite and $H$ is separable, then $\Gamma$ is residually finite \cite{BaumTret}.

\begin{remark}
	For any property $\mathcal{P}$, if $H$ is not $\mathcal{P}$-separable, then $\Gamma$ is not residually $\mathcal{P}$. Berlai has shown that for $\mathcal{P}$ the properties of solvability and amenability, so long as $G$ is residually $\mathcal{P}$, failure to be $\mathcal{P}$-separable is the only obstruction to $\Gamma$ being residually $\mathcal{P}$ \cite{Berlai}.
\end{remark}

\begin{lemma}\label{8} Suppose that $H<G$. Take $\overline{H}< \overline{G}$ to be another copy of our group-subgroup pair. If $H$ is not separable in $G$, then $\Gamma^{\prime}= \langle G, \overline{G}~|~a_it = \overline{a_it}, ~~1\leq i \leq k \rangle = G \ast_{H=\overline{H}} \overline{G}$ is not residually finite.
\end{lemma}

\begin{proof} We will show that if $g \in G-H$ is an element that cannot be separated from $H$ in finite quotients, that the non-trivial element $g^{-1}\overline{g} \in \Gamma^{\prime}$ is trivialized in every finite quotient, so $\Gamma^{\prime}$ is not residually finite. An arbitrary map $\Phi$ from $\Gamma^{\prime}$ to a finite group $Q$ will factor as a pair of maps $\psi,\phi: G \to Q$. By the definition of amalgamation, $\psi(h)=\phi(h)$ for all $h \in H$. To see that $g^{-1}\overline{g}$ is trivialized, we show that the functions $\phi$ and $\psi$ agree on $g$ as well. Construct $\Psi: G \to Q \times Q$, which is $(\psi, \phi)$. This new target group is still finite, and the image of $H$ is contained in the diagonal. As $\Psi(g) \in \Psi(H) \subset \Delta$, this implies that $\psi(g)=\phi(g)$. Therefore $\Phi(g^{-1}\overline{g}) =1$.\end{proof}

\begin{remark} If taking the direct product of groups preserves property $\mathcal{P}$ (eg. solvability, amenability), then if $H$ is not $\mathcal{P}$-separable in $G$, the same proof as above implies that $\Gamma^{\prime}=\langle G, \overline{G} ~|~ H = \overline{H} \rangle$ is not residually-$\mathcal{P}$. Kahrobaei has shown that for $\mathcal{P}$ the properties of solvability and amenability, so long as $G$ is residually $\mathcal{P}$, failure to be $\mathcal{P}$-separable is the only obstruction to  $\Gamma^{\prime}$ being residually $\mathcal{P}$\cite{DK}.
\end{remark}

We will use Lemmas \ref{7} and \ref{8} to show that the HNN extensions $\Gamma_k$ and the amalgamated products $\Gamma_k^{\prime}$ are not residually finite, by recognizing that the subgroup $H_k$ is not separable in $G_k$.

\begin{definition} A group $G$ is Hopfian if every surjective endomorphism of $G$ is an automorphism. 
\end{definition}
 
 In 1971 Gilbert Baumslag proved in \cite{BG}:
\begin{lemma}\label{Hopf}Finitely generated free-by-cyclic groups are residually finite, and therefore Hopfian. 
\end{lemma}
In particular, as the groups $G_k$ are free-by-cyclic, they are Hopfian, so to check that endomorphisms are automorphisms we need only check that they are surjective.
\end{section}

\begin{section}{$H_k$ is not a separable subgroup of $G_k$}
\begin{lemma}
$H_2$ is not a separable subgroup of $G_2$.
\end{lemma}
\begin{proof}
To show that $G_2$ is not $H_2$-separable, we use the result of Burns, Karrass, and Solitar \cite{BKS} that
$$H_{BKS}=\langle \alpha^{-1},~ y\alpha^{-1}y^{-2}\alpha \rangle ~~\leq~~ G_{BKS}= \langle \alpha, \beta, y~|~\alpha^{y} = \alpha\beta, \beta^{y} = \beta \rangle$$ is not separable. We demonstrate an automorphism $\phi$ of $G_2$ from which the isomorphism carrying $(G_2, \phi(H_2))$ to $(G_{BKS}, H_{BKS})$ is clear.
Recall the presentation 
 $$H_2= \langle a_1t, a_2t \rangle ~~ \leq ~~ G_2= \langle a_1, a_2, t ~|~ a_1^t=a_1,~ a_2^t=a_2a_1 \rangle= \langle a_1, a_2, t ~|~ [a_1, t]=1,~ [a_2, t]=a_1 \rangle.$$
 Consider the map:
 \begin{align*}
  a_2 &\mapsto a_2^{-2}ta_2\\
  t~ &\mapsto a_2^{-1}t^{-1}a_2=a_1t^{-1}\\
  a_1 &\mapsto a_1.
 \end{align*} 
 We verify that $\phi$ is an endomorphism: 
 \begin{align*}
	 [\phi(a_1), \phi(t)] &=[a_1,  a_1t^{-1}]=a_1^{-1}ta_1^{-1}a_1a_1t^{-1}=1= \phi([a_1, t]), \\
	 [\phi(a_2), \phi(t)]&= [a_2^{-2}ta_2, a_2^{-1}t^{-1}a_2]=[a_2^{-1}t,~ t^{-1}]^{a_2}= a_2^{-1}(t^{-1}a_2ta_2^{-1}tt^{-1})a_2= [a_2,t]=a_1 = \phi(a_1).
 \end{align*}
$\phi$ is also surjective: \begin{align*}
a_1&=\phi(a_1)\\
a_2&=\phi((a_2t)^{-1})\\
t~&=\phi(t^{-1}a_1)
\end{align*} Because $G_2$ is Hopfian, and $\phi$ is a surjective endomorphism, it must be an automorphism. The image of $a_1t$ is $$a_1a_2^{-1}t^{-1}a_2= [a_2, t]a_2^{-1}t^{-1}a_2= a_2^{-1}t^{-1}a_2ta_2^{-1}t^{-1}a_2 = (a_2^{-1}t^{-1}a_2)t(a_2^{-1}t^{-1}a_2),$$ and noting that $t$ and $a_2^{-1}t^{-1}a_2$ commute, we can write $\phi(a_1t)=ta_2^{-1}t^{-2}a_2$. The image of $a_2t$ is $a_2^{-1}$. Therefore $${\phi(H_2) = \langle a_2^{-1}, ta_2^{-1}t^{-2}a_2 \rangle}.$$ The isomorphism between $(G_2, \phi(H_2))$ and $(G_{BKS}, H_{BKS})$ is apparent from their definitions. By \cite{BKS}, $H_{BKS}$ is not a separable subgroup of $G_{BKS}$, and therefore $H_2$ is not a separable subgroup of $G_2$.\end{proof}

Next we show that the non-separability of $H_k$ in $G_k$ follows from the non-separability of $H_2$ in $G_2$. There is a natural inclusion $G_2 \hookrightarrow G_k$ which is just $a_1 \mapsto a_1, a_2 \mapsto a_2, t \mapsto t$. In the following we will abuse notation and write $G_2$ and $H_2$ for the image in $G_k$ of $G_2$ and $H_2$ under the inclusion. Dison and Riley develop a description of elements of $G_2 \cap H_k$ in \cite{Tim}, which we include here for the convenience of the reader:\\
Assign an order to the elements $\{a_1,\dots, a_k \}$, which we will call `priority': $a_{i+1}> a_{i}$ for all $i$. 
\begin{definition} The piece decomposition of a word $u \in F_k = \langle a_1, \dots, a_k \rangle$, is a grouping $u\cong \pi_1 \cdots \pi_l$ where $\pi_i$ are maximal words without occurrences of $a_k^{\pm 1}$, except possibly with prefix $a_k$ or suffix $a_k^{-1}$.	
\end{definition}

For example, $a_1a_2^{-1}a_1^2a_2a_1^{-1}a_2a_1^{2}a_2^{-2}$ has piece decomposition $(a_1a_2^{-1})(a_1^2)(a_2a_1^{-1})(a_2a_1^2a_2^{-1})(a_2^{-1}),$ where the parentheses indicate the different pieces. This piece-decomposition can be recursively defined. In particular, words containing no $a_k$ can be broken into pieces with respect to the next highest priority letter occurring. 

\begin{lemma}[{Dison and Riley}]\label{Key}
  A word $w = t^r u$ represents an element of $H_k$ if and only if $u$ has piece decomposition $u = \pi_1 \cdots  \pi_l$ and these pieces satisfy that for $p_0 = r$, there is a $p_i$ such that $t^{p_i} \pi_{i + 1} \in H_k t^{p_{i + 1}}$ and $p_l = 0$. When $p_{i+1}$ exists satisfying $t^{p_i} \pi_{i + 1} \in H_k t^{p_{i + 1}}$, it is unique. \end{lemma}

\begin{lemma}[{Dison and Riley}]\label{intersect}
	$G_2 \cap H_k = H_2$.
\end{lemma}
\begin{proof} From the definition, it is clear that $H_2 \subset G_2 \cap H_k$. Suppose that $w \in G_2 \cap H_k$. Using the free-by-cyclic normal form for $G_k$, rewrite $w= t^ru$ where $u \in \langle a_1,\dots, a_k \rangle$. Observe that as $w \in G_2$, this is also in the normal form for $G_2$, so $u \in \langle a_{1}, a_{2} \rangle$. By Lemma~\ref{Key}, this word is in $H_k$ if and only if $u$ has a piece decomposition $u = \pi_1 \cdots \pi_l$ and a tuple $(p_0=r,p_1,\dots, p_{l-1}, p_l=0)$, such that $t^{p_i}\pi_{i+1} \in H_kt^{p_{i+1}}$. The maximum priority letter that can occur in a word in $G_2 \cap H_k$ is $a_2$. As $t^{p_i}\pi_{i+1} \in H_kt^{p_{i+1}}$ contains no letters of priority greater than 2, we get that $t^{p_i}\pi_{i+1} \in H_2t^{p_{i+1}}$ for all $i$. From Lemma~\ref{Key} we have that $w \in H_2$. Therefore $G_2 \cap H_k = H_2$.\end{proof}

\begin{lemma}\label{simp} The intersection of a subgroup $H<G$ with a separable subgroup $S<G$ is separable in $H$.
\end{lemma}

\begin{proof}
As $S$ is separable, $\displaystyle S=\hspace{-.3cm}\bigcap_{\substack{S\leq K<G \\ \left[G:K\right] < \infty}}\hspace{-.3cm} K$. So $$S \cap H = \left( \bigcap_{\substack{S\leq K<G \\ \left[G:K\right] < \infty}}\hspace{-.3cm} K \right)\cap H = \bigcap_{\substack{S\leq K<G \\ \left[G:K\right] < \infty}}\hspace{-.2cm} \left(K \cap H\right)$$ and we have expressed $S \cap H$ as the intersection of a family of finite-index subgroups in $H$, since $K \cap H$ is finite index in $H$.\end{proof}

\begin{proof}[Proof of Lemma \ref{notsep}]
	By Lemma~\ref{simp}, if $G_2 \cap H_k$ is not separable in $G_2$, then $H_k$ is not separable in $G_k$. By Lemma~\ref{intersect}, $G_2 \cap H_k=H_2$, so since $H_2$ is not separable in $G_2$, $H_k$ is not separable in $G_k$.\end{proof}

Lemmas \ref{7} and \ref{8} showed that separability of $H_k$ in $G_k$ is a necessary condition for residual finiteness of $\Gamma_k$ and $\Gamma^{\prime}_k$, so the non-separability of $H_k$ in $G_k$ shown in Lemma \ref{notsep} implies Theorems \ref{notrf1} and \ref{notrf}.
\end{section}

\begin{section}{Generalizations of the Hydra Groups}
In the last section we saw that for every $k$, $H_k$ is not a separable subgroup of $G_k$. We were interested in these groups because Dison and Riley showed that $H_k$ is distorted like the Ackermann function $A_k$ in $G_k$, which forces the Dehn function of the doubles to be large. In this section we consider other pairs for which the machinery of Dison and Riley show that the analogous HNN extensions and free products with amalgamation will have exponential or superexponential Dehn function. We will show that these groups too are not residually finite.\\

The following proposition is an extension of the example of Burns, Karrass, and Solitar in \cite{BKS}. It is the key to proving Theorem \ref{genHydra}: the subgroup $H_k$ is separable in the generalized hydra group $G_k(\textbf{w})$ only when $G_k(\textbf{w}) = F_k \times\mathbb{Z}$.

\begin{prop}\label{BK} If $r>0$, the subgroup $H_2(r,1)=\langle a_1t^r, a_2t \rangle <G_2$ is not separable.
\end{prop}

\begin{remark}
For every $s \in \mathbb{Z}$ there is an automorphism, $\eta_s$ of $G_2$ which carries $H_2(r,0)$ to $H_2(r,s)$, defined by $\eta_s(a_2) = a_2t^{s-1},$ $\eta_s(t) = t$. Observe that $\eta_s(a_1) = \eta_s([a_2, t]) = [a_2t^s, t] = (a_2a_1^{s})^{-1}(a_2a_1^{s+1}) = a_1$. Below we will actually prove that $H_2(r,0)$ is not separable.
\end{remark}

\begin{remark}
We restrict to the case $r> 0$ in order to be able to apply the techniques of Dison and Riley. In particular, we wish to have the analogue of Lemma \ref{Key} in order to prove Lemma \ref{notinH}.
\end{remark}

\begin{remark} If there was an automorphism $\phi$ of $G_2$ carrying $H_2$ onto $H_2(r,0)$, the result would follow immediately. Suppose that there was such an automorphism $\phi$. Let $q: G_2 \to G_2^{ab} = \langle a_2, t~|~ [a_2,t]=1\rangle$ be the abelianization map. The automorphism $\phi: G_2 \to G_2$ descends to $\phi_{ab}$, an automorphism of the abelianization $G_2^{ab}$. The restriction of $\phi$ to $H_2$, called $\phi^{res}$, will descend to an isomorphism from $q(H_2)$ to $q(H_2(r, 0))$, which agrees with the restriction to $q(H_2)$ of $\phi_{ab}$. Note that $q(H_2) = q(G_2)$, and $q(H_2(r, 0))= \langle t^r, a_2 \rangle$. When $r>1$, this is a proper subgroup of $G_2^{ab}$.  This is a contradiction. Because $\phi_{ab}$ and $\phi_{ab}^{res}$ have the same domain and $\phi_{ab}^{res}$ is a restriction of $\phi_{ab}$, they should be the same function. However, these maps have different ranges. Therefore the proof of the proposition requires more than an application of Lemma \ref{notsep}.
\end{remark}
\begin{lemma}\label{tnotinH} $\langle t \rangle \cap H_2(r,0) = \{1\}$
\end{lemma}
\begin{proof} Suppose for the contradiction that there is a non-trivial element of the intersection. It can be expressed either as an element of $H_2(r,0)$ or as an element of $\langle t \rangle$: $$(a_1t^r)^{\alpha_1}a_2^{\beta_1}\cdots (a_1t^r)^{\alpha_n}a_2^{\beta_n}=t^m$$ 
for some $\alpha_i, \beta_i$ and $m \neq 0$ . There is a van Kampen diagram for the word $w=(a_1t^r)^{\alpha_1}a_2^{\beta_1}\cdots (a_1t^r)^{\alpha_n}a_2^{\beta_n}t^{-m}$ over the $G_2$ presentation: $\langle a_1, a_2, t ~|~ a_1^{t}=a_1,~a_2^t=a_2a_1\rangle = \langle a_1, a_2, t ~|~ a_1^{t}=a_1,~t^{a_2}=ta_1^{-1}\rangle$. From the second presentation it is clear that there are $a_2$ corridors in any van Kampen diagram for which the boundary word contains either an $a_2$ or $a_2^{-1}$. See Bridson and Gersten \cite{BG} for a detailed account of corridors in van Kampen diagrams. Since the word $w$ contains the letter $a_2$, there are $a_2$-corridors. A corridor is innermost if the boundary word it cuts off is a word on only the generators $a_1$ and $t$. There are always at least two innermost corridors, so at least one of them will cut off a word $\delta$, which is of the form $(a_1t^r)^{\alpha_i}$ for some $i$. The word along the side of the $a_2$ corridor will either be a power $t^ka_1^{-k}$ or $t^k$, call it $\gamma$. The word $\delta \gamma^{-1}=1$, but we note that this equality is not possible, as there is a non-zero index sum of either $a_1$ or $t$ in $\delta \gamma^{-1}$. \end{proof}

\begin{lemma}\label{notinH} The word $[t^{-1}, a_2^{-2}t^{-1}a_2^{2}] ~\not\in H_2(r,0)$ when $r \geq 1$.
\end{lemma}
\begin{proof} The analogue of Lemma~\ref{Key} holds for $H_2(r,0)$. That is, we can decide whether or not a word in $\langle a_1, a_2 \rangle$ is in $H_2(r,0)$ by considering whether the rewriting can be carried out on each successive piece. 
\begin{align*}
ta_2^{-2}ta_2^2t^{-1}a_2^{-2}t^{-1}a_2^2 =(a_2a_1^{-1})^{-2}(a_2a_1^{-2})^2(a_2a_1^{-1})^{-2}a_2^2 =(a_1a_2^{-1})(a_1^{-1})(a_2a_1^{-1}a_2^{-1})(a_1)(a_2)
\end{align*}
where the parentheses in the final line separate the different pieces. There is no $p$ such that $a_1a_2^{-1} \in H_2(r,0)t^{-p}$. Suppose that there is. Then 
\begin{align*}
ht^{-p}=a_1a_2^{-1} ~\Rightarrow ~ h= a_1a_2^{-1}t^p= t^pa_1^{1-p}a_2^{-1}~\Rightarrow ~ t^pa_1^{1-p}=ha_2 \in H_2(r,0) .
\end{align*}
We can rewrite $t^pa_1^{1-p} \in H_2(r,0)$, as $(a_1t^r)^{1-p}t^{r(p-1)+p} \in H_2(r,0)$, and since $a_1t^r \in H_2(r,0)$, it follows that $t^{r(p-1)+p} \in H_2(r,0)$.
Lemma \ref{tnotinH} implies that $r(p-1)+p=0$, so $p(r+1)=r$. Since $r>0$, we get $p=\dfrac{r}{1+r}$, which is not an integer. Thus $ [t^{-1},~ a_2^{-2}t^{-1}a_2^{2}] \not \in H_2(r,0).$ \end{proof}

\begin{proof}[Proof of Proposition \ref{BK}] For $r > 0$, the arguments of Burns, Karrass, and Solitar can be translated directly to work for this easy variation of their example \cite{BKS}. For the convenience of the reader, we repeat their argument (almost) verbatim. We drop all decoration and use $H$ to refer to $H_2(r,0)$ throughout this proof.\\

Let $\mathcal{T}$ be the infinitely generated group $\langle t_k~|~[t_k,t_{k+1}]=1,~ k \in \mathbb{Z}\rangle$. To make calculations easier, Burns, Karrass, and Solitar rewrite $G_2$ as the HNN extension $G_2=\langle \mathcal{T},~a_2~|~k \in \mathbb{Z},~ t_k^{a_2} =t_{k+1}\rangle$. In our original presentation, $t_k=t^{a_2^k}$. This implies that $a_1 = [a_2, t] = a_2^{-1}t^{-1}a_2 t = t_1^{-1}t_0$, and the word $[t^{-1}, a_2^{-2}t^{-1}a_2^{2}]=[t_0^{-1}, t_2^{-1}]$.\\

Given an arbitrary finite-index subgroup $L$ satisfying $H< L$, Burns, Karrass, and Solitar find a subgroup $N < L \cap \mathcal{T}$ such that
$$H \cap \mathcal{T}< N \triangleleft~ \mathcal{T}.$$ 
Analysis of the quotient $\mathcal{T}/N$ will imply that the word $[t^{-1},~a_2^{-2}t^{-1}a_2^{2}]$ is contained in $N\subset L$. According to Lemma \ref{notinH}, $[t^{-1}, a_2^{-2}t^{-1}a_2^{2}]$ is not an element of $H$. As $L$ is an arbitrary finite index subgroup containing $H$, this implies that $H$ is not separable. \\

If $L$ is a finite-index subgroup $L <G$, the core of $L$, $\text{core}(L) =\displaystyle \cap_{g \in G} L^g$, is a finite-index normal subgroup. Moreover, $\text{core}(L) \cap \mathcal{T}$ is still normal and finite index in $\mathcal{T}$. The group $N$ above is given by $(H \cap \mathcal{T})(\text{core}(L)\cap \mathcal{T})$. The majority of the work of this proof is in showing that $N$ is normal in $\mathcal{T}$.\\

\begin{lemma} $H \cap \mathcal{T} = \langle t_{i}^{-1}t_{i-1}^{r+1}~ |~ i \in \mathbb{Z} \rangle$.
\end{lemma}
\begin{proof}
Notice that all elements of $\mathcal{T}$ have trivial $a_2$ index sum, since every element in the generating set has zero $a_2$ index sum: $t_i = t^{a_2^i}$. The elements of $H$ with trivial $a_2$ index sum are all generated by $a_2$ conjugates of $a_1t^r$, so $H \cap \mathcal{T} \leq \langle(a_1t^r)^{a_2^{i-1}}~|~i\in \mathbb{Z}\rangle = \langle (t_1^{-1}t_0^{r+1})^{a_2^{i-1}} ~|~i\in \mathbb{Z} \rangle  = \langle t_{i}^{-1}t_{i-1}^{r+1}~|~i\in \mathbb{Z}\rangle$. The other inclusion is clear. \end{proof}

\begin{clma}\label{LemmaNormal}
	$N = (H \cap \mathcal{T})(\text{\textup{core}}(L) \cap \mathcal{T})$ is normal in $\mathcal{T}$.
\end{clma}

	Notice that $H \cap \mathcal{T}$ and $\text{core}(L)\cap \mathcal{T}$ are invariant under conjugation by $a_2$. Therefore $N$ is invariant under conjugation by $a_2$. We will next establish that $(H\cap\mathcal{T})^{t_0^{\pm 1}}\subset N$ by considering where conjugation by $t_0 ^{\pm 1}$ sends the generators $t_i^{-1}t_{i-1}^{r+1}$.

\begin{lemma}\label{more} If $i <0$, then $(t_i^{-1}t_{i-1}^{r+1})^{t_0^{\pm 1}} \in H \cap \mathcal{T}$.
\end{lemma}

\begin{proof} For $i<0$, both $t_0^{-1}t_i^{(r+1)^{i}} \in H \cap \mathcal{T}$ and $t_i^{(r+1)^{i}}t_0^{-1} \in H \cap \mathcal{T}$, as $$t_0^{-1}t_i^{(r+1)^{i}}= t_0^{-1}t_{-1}^{(r+1)}(t_{-1}^{-(r+1)}t_{-2}^{(r+1)^{2}}) \cdots (t_{i+1}^{-(r+1)^{i-1}}t_i^{(r+1)^i}) = t_0^{-1}t_{-1}^{(r+1)}(t_{-1}^{-1}t_{-2}^{(r+1)})^{(r+1)} \cdots (t_{i+1}^{-1}t_i^{r+1})^{(r+1)^{i-1}},$$ where the second equality holds since $[t_k, t_{k+1}]=1$ for all $k$. The same kind of rewriting shows $t_i^{(r+1)^{i}}t_0^{-1} \in H \cap \mathcal{T}$. \\

Next $(t_i^{-1}t_{i-1}^{r+1})^{t_0}$ can be rewritten using words of the form $t_0^{-1}t_i^{(r+1)^{i}}$:
$$t_0^{-1}t_i^{-1}t_{i-1}^{r+1}t_0 = t_0^{-1}t_i^{(r+1)^{i}}(t_i^{-(r+1)^{i}}t_i^{-1}t_{i-1}^{r+1}t_i^{(r+1)^{i}})t_i^{-(r+1)^{i}}t_0=(t_0^{-1}t_i^{(r+1)^{i}})(t_i^{-1}t_{i-1}^{r+1})(t_0^{-1}t_i^{(r+1)^{i}})^{-1}.$$
Since $t_i^{-1}t_{i-1}^{r+1}$ and $t_0^{-1}t_i^{(r+1)^{i}}$ are in $H \cap \mathcal{T}$, so is $(t_i^{-1}t_{i-1}^{r+1})^{t_0} \in H \cap \mathcal{T}$. Similarly, $$t_0t_i^{-1}t_{i-1}^{r+1}t_0^{-1}= (t_i^{{(r+1)}^i}t_0^{-1})^{-1}(t_i^{-1}t_{i-1}^{r+1})(t_i^{{(r+1)}^i}t_0^{-1})^{-1} \in H \cap \mathcal{T}.$$
 \end{proof}

For the other half of the generators, we can only show the following weaker lemma:
\begin{lemma}\label{less} If $l>0$, then $(t_{l}^{-1}t_{l-1}^{r+1})^{t_0^{\pm 1}} \in (H \cap \mathcal{T})(\text{\textup{core}}(L) \cap \mathcal{T})$.
\end{lemma}
\begin{proof}
Because $\text{core}(L) \cap \mathcal{T}$ is finite index in $\mathcal{T}$, there exists $t_it_j^{-1} \in \text{core}(L) \cap \mathcal{T}$ with $i-j< 0$. Indeed there are infinitely many generators and only finitely many cosets of $\text{core}(L) \cap \mathcal{T}$. Since $\text{core}(L) \cap \mathcal{T}$ is normal,
 ${t_{i-j}t_0^{-1}= (t_it_j^{-1})^{a_2^{-j}} \in \text{core}(L) \cap \mathcal{T}}$. By conjugating $t_{i-j}t_0^{-1}$ by $a_2^{(i-j)k}$ we get ${t_{(i-j)(k+1)}t_{(i-j)k}^{-1} \in \text{core}(L) \cap \mathcal{T}}$ for all $k \in \mathbb{Z}$. Stringing these elements together, we get that $t_{(i-j)k}t_0^{-1} \in \text{core}(L) \cap \mathcal{T}$ for all $k \in \mathbb{Z}$.

Given $l>0$, choose $n=(i-j)k$ such that $n>l$. Then
\begin{align*} t_0t_l^{-1}t_{l-1}^{r+1}t_0^{-1}= (t_0t_n^{-1})(t_n t_l^{-1}t_{l-1}^{r+1}t_n^{-1})(t_n t_0^{-1})= (t_0t_n^{-1})(t_0t_{l-n}^{-1}t_{l-1-n}^{r+1}t_0^{-1})^{a_2^n}(t_n t_0^{-1})
\end{align*} 

Lemma \ref{more} implies that the middle term is an element of $H\cap \mathcal{T}$, as $l-n<0$, and $H\cap \mathcal{T}$ is invariant under conjugation by $a_2$. The conjugating terms $t_n t_0^{-1} \in \text{core}(L) \cap \mathcal{T}$ and so $t_0t_l^{-1}t_{l-1}^{r+1}t_0^{-1} \in (H \cap \mathcal{T})(\text{core}(L) \cap \mathcal{T})$.\end{proof}

 From Lemmas \ref{more} and \ref{less}, we have that each of the generators of $H \cap \mathcal{T}$ is conjugated by $t_0$ and $t_0^{-1}$ into $N=(H \cap \mathcal{T})(\textup{core}(L) \cap \mathcal{T})$. From the normality of $\text{\textup{core}}(L) \cap \mathcal{T}$, we get $((H \cap \mathcal{T})(\text{\textup{core}}(L) \cap \mathcal{T}))^{t_0^{\pm 1}} = N$, and we can conjugate $N^{t_0^{\pm 1}} \subset N$ by $a_2^k$ for $k \in \mathbb{Z}$ to get $N^{t_k^{\pm 1}}\subset N$. Therefore $N$ is a normal subgroup of $\mathcal{T}$.\\

That $[t_0^{-1}, t_2^{-1}]$ is in $N$ follows easily from $N$ being normal in $\mathcal{T}$. Indeed, since $t_0^{-1}t_i^{(r+1)^i} \in H \cap \mathcal{T}$ for $i<0$, it follows that $t_i^{(r+1)^i}N = t_0 N$. In the quotient $\mathcal{T}/ N$, the images of $t_i$ and $t_0$ commute when $i<0$. When $i>0$, we can rewrite $[t_0, t_i]= [t_{-i}, t_0]^{a_2^{i}}= 1^{a_2^{i}}=1$, so they too commute in the quotient. Therefore $\mathcal{T}/ N$ is an abelian group and $[t_0^{-1}, t_2^{-1}] \in N$. Since $H$ and $\text{core}(L)$ are subgroups of $L$ and $N=(H \cap \mathcal{T})(\text{core}(L) \cap \mathcal{T})$, it follows that $N \leq L$. Therefore $[t_0^{-1}, t_2^{-1}]$ is an element of $L$ but not of $H_2(r,0)$, and so $H_2(r,0)$ is not a separable subgroup of $G_2$.\\
\end{proof}

Consider the group $G_k(\mathbf{w}) = \langle a_1, \dots, a_k, t~|~a_1^t=a_1w_1, \dots, a_k^t = a_kw_k \rangle$, where $\mathbf{w} = (w_1, \dots, w_k)$, with each $w_i$ a positive word on the generators $\{ a_1 \dots a_{i-1}\}$. Recall the statement of Theorem \ref{genHydra}: The subgroup $H_k=H_k(1,\dots,1)$ is separable in $G_k(\mathbf{w})$ if and only if $\mathbf{w}=(1,\dots, 1)$.\\

\begin{proof}[Proof of Theorem \ref{genHydra}] If $\mathbf{w}=(1, \dots, 1)$, then $G_k= F_k \times \langle t \rangle$. $G_k$ is subgroup separable, so in particular, $H_k$ is separable. If $\mathbf{w}\neq(1, \dots, 1)$, then $G_k(\mathbf{w})$ has $\mathbf{w}$ with initial segment of the form $(1,\dots, 1, w_c,\dots, w_k)$ where $w_c$ is the first non-trivial word. The subgroup $\langle w_c, a_c, t  \rangle$ is isomorphic to $G_2$ and the subgroup $\langle w_ct^{|w_c|},a_ct \rangle$ is isomorphic to the subgroup $H_2(|w_c|,1)$, where $|w_c|$ is the length of $w_c$ in $\langle a_1, \dots, a_{c-1} \rangle$. Proposition~\ref{BK} implies that this subgroup is not separable in $\langle w_c, a_c, t \rangle$. Since $\langle w_ct^{|w_c|},a_ct \rangle$ is the intersection $H_k(\mathbf{w}) \cap \langle w_c, a_c, t \rangle$, Lemma~\ref{simp} implies that $H_k$ is not separable in $G_k(\mathbf{w})$.\end{proof}

The most general form for which the methods of Dison and Riley apply are $H_k(\mathbf{r})\leq G_k(\mathbf{w})$. We are able to get only a partial characterization of separability in this case, which is a generalization of Theorem \ref{genHydra} and its proof.\\

Recall the statement of Theorem \ref{mostgenHydra}: Suppose that $\mathbf{w}= (1, \dots, 1, w_c, \dots, w_k)$ and $\mathbf{r}= (r_1, \dots, r_k)$, with conditions on $\mathbf{w}$, $\mathbf{r}$ as above. Let $[w_c]_i$ denote the index sum of $a_i$ in $w_c$. If $\sum_{i=1}^{c-1}[w_c]_i r_i \neq 0$, then $H_k(\mathbf{r})$ is not a separable subgroup of $G_k(\mathbf{w})$.\\

\begin{proof}[Proof of Theorem \ref{mostgenHydra}] We examine the subgroup $\langle a_c, w_c, t~|~w_c^t=w_c, a_c^t=a_cw \rangle$, which is isomorphic to $G_2$. The subgroup given by 
$$\langle w_ct^{\sum_{i=1}^{c-1}[w_c]_ir_i}, a_ct\rangle$$ is isomorphic to $H_2(\sum_{i=1}^c[w_c]_i r_i, 1)$, and by Lemma~\ref{BK}, this subgroup is not separable if $\sum_{i=1}^c[w_c]_i r_i \neq 0$.\end{proof}

\begin{remark}
Separability of $H_k$ in $G_k$ for the case that $\sum_{i=1}^c[w_c]_i r_i =0$ is not established by the argument above. The most basic examples for which our method fails are those of the form $$G_{c+1}(\mathbf{w})= \langle a_1, \dots, a_c, a_{c+1}, t~|~ a_i^t=a_i, i<c, a_{c}^t=a_cw_c, a_{c+1}^t=a_{c+1}w_{c+1} \rangle$$ and $$H_{c+1}(\mathbf{r})=\langle a_1, \dots, a_{c-1}, a_ct^{r_c}, a_{c+1} \rangle.$$ The simplest case of this failure is the group $G_3 = \langle a_1, a_2, t~|~a_1^t=a_1, a_2^t=a_2a_1, a_3^t=a_3a_2\rangle$ with subgroup $H_3(\mathbf{r})=\langle a_1, a_2t, a_3\rangle.$
\end{remark}
\end{section}
\bibliography{references}{}
\bibliographystyle{plain}
\end{document}